\documentclass[11pt]{article}

\usepackage[T1]{fontenc}
\usepackage{lmodern}
\usepackage{amsmath,amssymb,amsthm}
\usepackage[margin=1in]{geometry}
\usepackage{booktabs}
\usepackage[hidelinks]{hyperref}
\usepackage[nameinlink,noabbrev]{cleveref}

\newtheorem{theorem}{Theorem}
\newtheorem{corollary}[theorem]{Corollary}

\newcommand{\lavg}{\ell_{\mathrm{avg}}}
\newcommand{\doi}[1]{\href{https://doi.org/#1}{doi:#1}}

\title{Lean-Certified Infinite Counterexamples to Written on the Wall II Conjecture 194}
\author{Cameron Beeley}
\date{Version 1.2 (15 September 2026)}

\begin{document}
\maketitle

\begin{abstract}
For a finite simple graph $G$, let $\alpha(G)$ denote its independence number and let
\[
  \lavg(G)=\frac{1}{|V(G)|}\sum_{v\in V(G)}\alpha(G[N_G(v)])
\]
be the average independence number of its open neighbourhoods.  Written on the Wall II
Conjecture~194 asserts that every connected graph satisfying
$\alpha(G)\leq 1+\lavg(G)$ has a Hamiltonian path.  We give a four-parameter family of
counterexamples.  Its principal two-parameter subfamily satisfies the proposed inequality
with equality: for every pair of integers $s\geq 1$ and $t\geq 3$, it has $(s+1)t^2$
vertices, independence number $t+1$, average neighbourhood independence $\lavg=t$, and
minimum degree $s$, but has no Hamiltonian path.  This entire infinite subfamily is
machine-checked in Lean~4: one universally quantified theorem certifies its order,
connectivity, independence number, average neighbourhood independence, minimum degree,
conjecture hypothesis, and failure of traceability.  Thus no fixed lower bound on the
minimum degree repairs the conjecture.  The case $(s,t)=(1,3)$ has $18$ vertices, but the
formal certificate is parametric rather than a verification of that one graph alone.
\end{abstract}



\section{Introduction}

All graphs in this note are finite, simple, and undirected.  A graph is \emph{traceable} if
it has a Hamiltonian path.  For $v\in V(G)$, write $N_G(v)$ for the open neighbourhood of
$v$ and set
\[
  \ell_G(v)=\alpha(G[N_G(v)]),
  \qquad
  \lavg(G)=\frac{1}{|V(G)|}\sum_{v\in V(G)}\ell_G(v).
\]

Classical sufficient conditions for traceability often bound the independence number
against the connectivity.  The prototype is the theorem of Chv\'atal and Erd\H{o}s
\cite{chvatalerdos1972}, whose Hamiltonian-path corollary states that
$\alpha(G)\leq\kappa(G)+1$ implies that $G$ is traceable.  The Written on the Wall II list
places the conjecture studied here immediately beside that corollary: its Conjecture~195
is exactly $\alpha(G)-1\leq\kappa(G)\Rightarrow G$ traceable, and is recorded there as a
known theorem of Chv\'atal and Erd\H{o}s \cite{delavinaWOWII}.  Conjecture~194 is the same
inequality with the local average $\lavg(G)$ substituted for the connectivity $\kappa(G)$.
Our family shows that the substitution is unsound, and indicates why: $\lavg$ averages a
purely local quantity and is blind to the articulation structure that $\kappa$ controls.

The Graffiti.pc program generated many conjectured relations between graph invariants;
their development and the associated Written on the Wall II collection are described in
\cite{delavina2005history,delavina2007independence}.  The collection includes the following
sufficient condition for traceability \cite{delavinaWOWII}.

\begin{quote}
\textbf{Written on the Wall II Conjecture 194.}
If $G$ is a simple connected graph on $n>1$ vertices such that
$\alpha(G)\leq 1+\lavg(G)$, then $G$ has a Hamiltonian path.
\end{quote}

\noindent
We follow the source statement, writing $\alpha$ and $\lavg$ for the invariants it denotes
by $a$ and $l_{\mathrm{avg}}$.  The hypothesis $n>1$ is immaterial here: every graph we
construct has at least $18$ vertices.

We show that the conjecture fails for infinitely many graphs, including infinitely many
graphs on the proposed boundary $\alpha=1+\lavg$ and with arbitrarily large minimum degree.
The construction combines a dense join, which controls the local independence numbers, with
at least three end cliques, each attached to the rest of the graph through a single vertex.
These end blocks force distinct endpoints in any spanning path.  The complete two-parameter
equality family, including every invariant and obstruction used in the disproof, is certified
in Lean~4 at an immutable revision; the certificate is quantified over the parameters and is
not an enumeration of individual examples.

\section{The counterexample family}

For integers $m,r,t,s$ with $m,s\geq 1$ and $1\leq t\leq r$, define
$F^{(s)}_{m,r,t}$ as follows.  Take pairwise disjoint vertex sets
\[
  C=\{c_1,\dots,c_m\},\qquad
  H=\{h_1,\dots,h_r\},\qquad
  B_1,\dots,B_t,
\]
where $|B_i|=s$ for every $i$.
The set $C$ induces a clique, $H$ induces an independent set, and every vertex of $C$ is
adjacent to every vertex of $H$.  Each $B_i$ induces a clique and every vertex of $B_i$
is adjacent to $h_i$, with no other edges incident with $B_i$ outside $B_i$.  Equivalently,
$B_i\cup\{h_i\}$ induces a clique attached to the rest of the graph only through $h_i$.
When $s=1$, the sets $B_i$ are single pendant vertices and this is the original
three-parameter family.

\begin{theorem}\label{thm:family}
Let $s\geq 1$ and $3\leq t<r$, and put
\[
  M_s(r,t)=r^2-2r+\bigl(s(r-2)-1\bigr)t.
\]
For every $m\geq M_s(r,t)$, the connected graph $F^{(s)}_{m,r,t}$ satisfies
\[
  \alpha(F^{(s)}_{m,r,t})\leq 1+\lavg(F^{(s)}_{m,r,t})
\]
but has no Hamiltonian path.  Equality holds in the displayed inequality if and only if
$m=M_s(r,t)$.
\end{theorem}

\begin{proof}
The graph is connected because $C$ is nonempty and is completely joined to $H$.
We first determine its independence number.  An independent set meeting $C$ contains
at most one vertex of $C$, no vertex of $H$, and at most one vertex of each $B_i$; it
therefore has size at most $t+1\leq r$.  An independent set disjoint from $C$ contains at
most one vertex from each clique $B_i\cup\{h_i\}$, together with at most the $r-t$
remaining vertices of $H$.  Its size is therefore at most $r$, and this bound is attained
by $H$.  Hence
\[
  \alpha(F^{(s)}_{m,r,t})=r.
\]

The local independence numbers have four forms.  If $c\in C$, then
$N(c)=(C\setminus\{c\})\cup H$, so $\ell(c)=r$.  If $1\leq i\leq t$, then
$N(h_i)=C\cup B_i$ is the disjoint union of two cliques, so $\ell(h_i)=2$.  Each of the
other $r-t$ vertices of $H$ has neighbourhood $C$ and local independence number $1$.
Finally, if $b\in B_i$, then $N(b)=(B_i\setminus\{b\})\cup\{h_i\}$ is a clique, so
$\ell(b)=1$.  Consequently
\begin{equation}\label{eq:localsum}
  \sum_{v\in V(F^{(s)}_{m,r,t})}\ell(v)
  =mr+2t+(r-t)+st
  =mr+r+(s+1)t.
\end{equation}
Since $|V(F^{(s)}_{m,r,t})|=m+r+st$, the conjectured hypothesis is equivalent to
\begin{align*}
  r
  &\leq 1+\frac{mr+r+(s+1)t}{m+r+st}\\
  &\Longleftrightarrow
  (r-1)(m+r+st)\leq mr+r+(s+1)t\\
  &\Longleftrightarrow
  m\geq r^2-2r+\bigl(s(r-2)-1\bigr)t.
\end{align*}
This also shows that equality holds exactly when $m=M_s(r,t)$.

Finally, fix $i\leq t$.  The only edges from the nonempty set $B_i$ to its complement are
incident with $h_i$.  A spanning path must use at least one such edge.  It cannot use two:
both would exhaust the two path edges available at $h_i$, leaving no path edge from
$B_i\cup\{h_i\}$ to the nonempty remainder of the graph.  Thus exactly one edge of a
spanning path would cross from $B_i$ to its complement, forcing one endpoint of the path
to lie in $B_i$.  The pairwise disjoint sets $B_1,\dots,B_t$ would therefore require at
least $t\geq 3$ distinct endpoints, whereas a path has only two.  Hence
$F^{(s)}_{m,r,t}$ is not traceable.
\end{proof}

Setting $r=t+1$ gives a particularly transparent two-parameter equality family.

\begin{corollary}\label{cor:equality}
For integers $s\geq 1$ and $t\geq 3$, let
\[
  G_{s,t}=F^{(s)}_{(s+1)t(t-1)-1,\,t+1,\,t}.
\]
Then
\[
  |V(G_{s,t})|=(s+1)t^2,\qquad
  \alpha(G_{s,t})=t+1,\qquad
  \sum_{v\in V(G_{s,t})}\ell_{G_{s,t}}(v)=(s+1)t^3,
\]
and
\[
  \lavg(G_{s,t})=t,\qquad \delta(G_{s,t})=s.
\]
In particular, $\alpha(G_{s,t})=1+\lavg(G_{s,t})$, but $G_{s,t}$ has no
Hamiltonian path.
\end{corollary}

\begin{proof}
Substitution into \cref{thm:family} gives
\[
  M_s(t+1,t)
  =(t+1)^2-2(t+1)+\bigl(s(t-1)-1\bigr)t
  =(s+1)t(t-1)-1.
\]
The order is
\[
  (s+1)t(t-1)-1+(t+1)+st=(s+1)t^2,
\]
and \cref{eq:localsum} gives
\[
  \bigl((s+1)t(t-1)-1\bigr)(t+1)+(t+1)+(s+1)t
  =(s+1)t^3.
\]
Every vertex in a set $B_i$ has degree $s$.  The vertices in $H$ have degree $m$ or
$m+s$, and each vertex of $C$ has degree $m+t$.  Since
$m=(s+1)t(t-1)-1>s$ for $t\geq 3$, the minimum degree is $s$.  The remaining statements
follow from \cref{thm:family}.
\end{proof}

The family immediately rules out any repair based on a fixed minimum-degree threshold.

\begin{corollary}\label{cor:min-degree}
For every integer $d\geq 1$, there are infinitely many connected graphs $G$ with
$\delta(G)=d$ and $\alpha(G)=1+\lavg(G)$ that have no Hamiltonian path.
For the subfamily with $t=3$, one moreover has
\[
  |V(G_{d,3})|=9(d+1),\qquad \delta(G_{d,3})=\frac{|V(G_{d,3})|}{9}-1.
\]
\end{corollary}

\begin{proof}
Take $G=G_{d,t}$ in \cref{cor:equality} and let $t$ range over the integers at least $3$.
The displayed relation for $t=3$ follows from $|V(G_{d,3})|=(d+1)3^2$.
\end{proof}

Some small members of the equality family are shown in \cref{tab:first}.

\begin{table}[ht]
  \centering
  \begin{tabular}{@{}ccccccc@{}}
    \toprule
    $s$ & $t$ & $|C|$ & $|V(G_{s,t})|$ & $\delta$ & $\alpha$ & $\lavg$ \\
    \midrule
    $1$ & $3$ & $11$ & $18$ & $1$ & $4$ & $3$ \\
    $2$ & $3$ & $17$ & $27$ & $2$ & $4$ & $3$ \\
    $3$ & $3$ & $23$ & $36$ & $3$ & $4$ & $3$ \\
    $1$ & $4$ & $23$ & $32$ & $1$ & $5$ & $4$ \\
    $2$ & $4$ & $35$ & $48$ & $2$ & $5$ & $4$ \\
    $5$ & $3$ & $35$ & $54$ & $5$ & $4$ & $3$ \\
    \bottomrule
  \end{tabular}
  \caption{Sample graphs in the equality family of \cref{cor:equality}.}
  \label{tab:first}
\end{table}

\section{A machine-checked certificate for the infinite family}\label{sec:lean}

Conjecture~194 is catalogued in the Formal Conjectures benchmark
\cite{formalconjecturesrepo,formalconjecturespaper}, a Lean~4 collection of open and
solved research statements, where it was carried with the attribute
\texttt{category research open} until the correction recorded below.  Related Lean-certified
work on the same collection includes proofs of Conjectures 141--143 \cite{ferudun2026}.  The Lean~4 development \cite{lean4}, built on
\texttt{mathlib} \cite{mathlib}, at the immutable commit recorded in
\cite{formalconjectures194} formalises $G_{s,t}$ directly.  Its theorem
\texttt{family\_certificate} is universally quantified over natural numbers $s$ and $t$.
From hypotheses $1\leq s$ and $3\leq t$, it proves in a single conjunction that
\[
  |V(G_{s,t})|=(s+1)t^2,\quad G_{s,t}\text{ is connected},\quad
  \alpha(G_{s,t})=t+1,
\]
\[
  \lavg(G_{s,t})=t,\quad \delta(G_{s,t})=s,\quad
  \alpha(G_{s,t})\leq 1+\lavg(G_{s,t}),
\]
and that $G_{s,t}$ has no Hamiltonian path.  In particular, the formalisation certifies
every member of the infinite equality family in \cref{cor:equality}, not only a finite
list of examples.  Equality in the conjectured bound follows immediately inside the
formal theory from the certified identities $\alpha=t+1$ and $\lavg=t$.

The same immutable revision contains a separate theorem \texttt{conjecture194} proving
the repository's exact proposition with \texttt{answer(False)}.  It specialises the
construction to $G_{1,3}$, the $18$-vertex graph consisting of an $11$-vertex clique,
four independent vertices joined completely to the clique, and three pendant vertices
attached to distinct vertices on the independent side.  Both proof files are free of
\texttt{sorry} and \texttt{admit} and compile with the project's pinned Lean~4.27.0
toolchain.

Thus the infinite equality family, the disproof of Conjecture~194, and the arbitrarily
large minimum-degree consequence are all covered by the parametric certificate.  The
broader four-parameter envelope in \cref{thm:family} is an additional direct argument in
this note and has not been formalised in full; it is not needed for any of those three
headline conclusions.

\section{Concluding remarks}

The Lean-certified equality examples show that replacing the weak inequality in
Conjecture~194 by an equality does not repair the statement.  More generally, increasing
$m$ beyond $M_s(r,t)$ makes the conjectured inequality strict while leaving the end-block
obstruction unchanged.  Thus the failure is not confined to a single boundary example.

The construction indicates that any repaired sufficient condition involving
$\alpha(G)$ and $\lavg(G)$ must also control a structural obstruction to traceability;
the two averaged invariants alone do not prevent arbitrarily many end blocks.  In
particular, \cref{cor:min-degree} shows that no fixed minimum-degree assumption suffices.
A condition controlling articulation structure, rather than merely vertex degrees, would
be required to exclude this family.  No claim is made here that $G_{1,3}$ is a smallest
counterexample among all graphs.

\section*{Code availability}
The immutable Lean source, including the parametric family certificate and the direct
formal negation of Conjecture~194, is linked in \cite{formalconjectures194} and archived
at \doi{10.5281/zenodo.22779981}.  The
corresponding correction to the upstream Formal Conjectures repository was merged on
7~August 2026 \cite{formalconjecturesPR}; that repository now records Conjecture~194 as
solved, with the formalised implication marked \texttt{answer(False)} and attributed to
the immutable revision cited above.

\section*{Declaration of generative AI use}
During the preparation of this work the author used OpenAI Codex to assist with exploratory
computation, Lean formalisation, source discovery, algebra checking, and language editing.
After using this tool, the author reviewed and edited the content as needed and takes full
responsibility for the content of the publication.

\end{document}